\documentclass[12pt]{article}

\usepackage{amsthm,tikz,amssymb, latexsym, amsmath, amscd, amsfonts, array, graphicx, lmodern, slantsc,enumerate,stmaryrd,rotating}

\usepackage[english]{babel}        
\usepackage[all]{xy}
\CompileMatrices

\usepackage{color}
\def\red{}%\textcolor{red}}

\def\colim{\protect\operatorname{colim}}
\def\hocolim{\protect\operatorname{hocolim}}

\def\F{\protect\operatorname{Conf}}

\def\UF{\protect\operatorname{UConf}}

\newtheorem{proposition}{Proposition}[section]
\newtheorem{corollary}[proposition]{Corollary}

\newtheorem{theorem}[proposition]{Theorem}
\newtheorem{remark}[proposition]{Remark}
\newtheorem{example}[proposition]{Example}
\newtheorem{lemma}[proposition]{Lemma}

\begin{document}

\title{Equivariant Nerve Lemma, simplicial difference, and models for configuration spaces on simplicial complexes}
\author{Emilio J.~Gonz\'alez and Jes\'us Gonz\'alez}
\date{\empty}

\maketitle

\begin{abstract}
Wiltshire-Gordon has introduced a homotopy model for ordered configuration spaces on a given simplicial complex. That author asserts that, after a suitable subdivision, his model also works for unordered configuration spaces. We supply details justifying Wiltshire-Gordon's assertion and, more importantly, uncover the equivariant properties of his more-general simplicial-difference model for the complement of a subcomplex inside a larger complex. This is achieved by proving an equivariant version of the Nerve Lemma. In addition, in the case of configuration spaces, we show that a slight variation of the model has better properties: it is regular and sits inside the configuration space as a strong and equivariant deformation retract. Our variant for the configuration-space model comes from a comparison, in the equivariant setting, between Wiltshire's simplicial difference and a well known model for the complement of a full subcomplex on a simplicial complex.
\end{abstract}



\section{Introduction and discussion of main results}
Let $A$ be a subcomplex of an abstract simplicial complex~$X$. An abstract simplicial complex $X\ominus A$, called the simplicial difference of $X$ and $A$, has been introduced in~\cite{MR3887191} to capture the homotopy type of the complement of the realization of $A$ inside that of $X$. Vertices of $X\ominus A$ are given by the minimal elements (with respect to face order) of the simplices in $X$ not in $A$, while a finite collection of such minimal simplices forms a simplex of $X\ominus A$ whenever their union yields a simplex in $X$. As shown by Whiltshire-Gordon through an application of the Nerve Lemma (see~Section~\ref{secW} for a review of the details), the realization $|X\ominus A|$ has the homotopy type of $|X|-|A|$. The simplicial-difference model is applied in~\cite{MR3887191} to the case of the ordered product $X^n$ (with respect to some linear order on the vertices of $X$) and its subcomplex $F_n$ corresponding to the fat diagonal. The resulting simplicial homotopy-model $X^n\ominus F_n$ for the ordered configuration space $\F(|X|,n)$ of $n$ non-colliding particles in $|X|$ is denoted by $C(X,n)$.

The simplicial action on $X^n$ of the symmetric group~$\mathcal{S}_n$ yields an action on the vertices of $C(X,n)$ and, indeed, on the whole simplicial model. Wiltshire-Gordon asserts without explicit argumentation  that
\begin{equation}\label{desliz}
\mbox{\emph{``under geometric realization, this action matches the usual $\mathcal{S}_n$-action on $\F(|X|,n)$.''}}
\end{equation}
\red{As noted in~\cite{MR3887191},} under such a basis it is enough to take a suitable subdivision $\overline{C}(X,n)$ of $C(X,n)$ regularizing the $\mathcal{S}_n$-action in order to get a simplicial quotient $\overline{C}(X,n)/\mathcal{S}_n$ modeling the homotopy type of the configuration space $\UF(|X|,n)$ of $n$ unordered non-colliding particles on $|X|$.

\red{A first goal of this paper is to supply explicit details supporting~(\ref{desliz}) and, in particular, the fundamental-group calculations of unordered configuration spaces in~\cite{MR3887191}. More generally, we show that Wiltshire-Gordon's simplicial difference $Y\ominus B$ captures the $G$-equivariant homotopy type of $|Y|-|B|$ as long as a finite group $G$ acts simplicially on a pair $(Y,B)$ of (abstract) complexes (Theorem~\ref{equiW}). This is established by proving an equivariant form of the Nerve Lemma (Theorem~\ref{conejura}).}

\red{A second goal of this paper is to show that, in the case of configuration spaces, a slight variant of Wiltshire-Gordon model has stronger properties, similar to those holding in the case of Abrams' (cubical) model for configuration spaces on graphs. Namely, for an abstract complex $X$, our model for $\F(|X|,n)$ is $\mathcal{S}_n$-regular and its geometric realization sits inside $\F(|X|,n)$ as a strong and $\mathcal{S}_n$-equivariant deformation retract (Theorem~\ref{maintheorem}).}

\red{Our variant arises in fact} from the following straightening of a mistaken assertion in Wiltshire-Gordon's work. (The root of the mistake is pinpointed in Remark~\ref{otropedo} below.) Remark~4.3 in~\cite{MR3887191} incorrectly claims that, in a certain concrete example, $X\ominus A$ differs from another well-known simplicial homo\-topy-model for $|X|-|A|$, which we now recall. As shown in~\cite[Lemma~70.1]{MR755006}, if $A$ is a \emph{full} subcomplex of X, then $|X|-|A|$ (linearly) strong deformation retracts to the subcomplex $C_{X,A}$ of $X$ spanned by the vertices of $X$ that fail to be vertices of $A$. The full subcomplex condition is not a theoretical issue, as this can be attained by suitable subdivision. For instance, with $X=\Delta^2$, the standard 2-simplex, and $A=\partial\Delta^2$, its (non-full) boundary (i.e., the example considered in~\cite[Remark~4.3]{MR3887191}), we see that, after a single barycentric subdivision ($bs$), $C_{bs(X),bs(A)}$ reduces to a single vertex, thus agreeing with (both $X\ominus A$ and) $bs(X)\ominus bs(A)$. Of course, this is a completely general (and elementary) phenomenon: As long as $A$ is full, minimal simplices of $X$ not belonging to $A$ are necessarily 0-dimensional. Consequently, from their bare definitions, $X\ominus A=C_{X,A}$, which then sits inside $|X|-|A|$ as a strong deformation retract. In Section~\ref{secM}, we review the construction of the deformation retraction, observing its compatibility with any given simplicial $G$-action on $X$ that restricts to one on $A$.

As detailed in Section~\ref{secT}, the above discussion sets the grounds for:
\begin{theorem}\label{maintheorem}
For an abstract simplicial complex $X$, the simplicial difference $bs(X^n)\ominus bs(F_n)$ agrees with
$$
C_{bs}(X,n):=C_{bs(X^n),bs(F_n)},
$$
and its geometric realization sits inside $\F(|X|,n)$ as a strong and $\mathcal{S}_n$-equivariant deformation retract. Furthermore, the $\mathcal{S}_n$-action on $C_{bs}(X,n)$ is regular, so the geometric realization of the simplicial quotient $$UC_{sb}(X,n):=C_{sb}(X,n)/\mathcal{S}_n$$ is  homotopy equivalent to $\UF(|X|,n)$.
\end{theorem}

We think of $C_{bs}(X,n)$ and $UC_{bs}(X,n)$ as (simplicial) higher-dimen\-sional analogues of the (cubical) model introduced by Abrams in~\cite{MR2701024} when $X$ is a graph. Not only Theorem~\ref{maintheorem} recasts the nice properties in Abrams' model, but the use of a barycentric subdivision in Theorem~\ref{maintheorem} plays the role of the assumption in Abrams' model that the graph should be sufficiently subdivided.

\begin{remark}\label{lddreg}{\em
Regularity of general $G$-complexes can be assured after a double barycentric subdivison~(\cite[Proposition~III.1.1]{MR0413144}). Yet, Theorem~\ref{maintheorem} guarantees that a single barycentric subdivision of $X^n$ suffices.
}\end{remark}

The use of barycentric subdivisions in Theorem~\ref{maintheorem} is not mandatory, and can be avoided in specific situations to get more efficient simplicial models without sacrificing on regularity or equivariant-deformation properties. Relevant instances of this philosophy are illustrated in Remarks~\ref{otropedo} and~\ref{recomendacion} and in Example~\ref{concreto} below. The reader might want to keep in mind that, for some purposes (e.g.~dimensional matters), subdivisions might turn out to be convenient (see Remark~\ref{mejor} below).

\begin{remark}\label{otropedo}{\em
The analysis of $C_{bs(X),A}$ in~\cite[Remark~4.3]{MR3887191} is not well-founded because, in that situation, $A$ fails to be a subcomplex of $bs(X)$. Yet, the idea could lead to more efficient subdivisions to be used in specific instances of Theorem~\ref{maintheorem}. A clean way to illustrate this is precisely in terms of the example in~\cite[Remark~4.3]{MR3887191}, where $X=\Delta^2$ and $A=\partial\Delta^2$. Namely, $A$ is a full subcomplex of the more efficient subdivision $sd(X)$ of $X$ shown in Figure~\ref{subdivision}. Of course, in that case $|C_{sd(X),A}|=|sd(X)\ominus A|$, which sits inside $|sd(X)|-|A|=|X|-|A|$ as a strong deformation retract. 
}\end{remark}

\begin{figure}[h!]
\begin{center}
\begin{tikzpicture}[scale=0.65]
\filldraw[black] (0,0) circle (1.5pt) node[right] {};
\filldraw[black] (3,0) circle (1.5pt) node[right] {};
\filldraw[black] (1.5,2.5) circle (1.5pt) node[right] {};
\filldraw[black] (1.5,.9) circle (1.5pt) node[right] {};
\draw[-] (0,0,0) -- (1.5,.9,0) node[right]{};
\draw[-] (3,0,0) -- (1.5,.9,0) node[right]{};
\draw[-] (1.5,2.5,0) -- (1.5,.9,0) node[right]{};
\draw[-] (0,0,0) -- (3,0,0) node[right]{};
\draw[-] (0,0,0) -- (1.5,2.5,0) node[right]{};
\draw[-] (3,0,0) -- (1.5,2.5,0) node[right]{};
\end{tikzpicture}
\end{center}
\caption{$sd(X)$}
\label{subdivision}
\end{figure}

\begin{remark}\label{recomendacion}{\em
For any abstract simplicial complex $X$, the (fat) diagonal $F_2$ is easily seen to be a full subcomplex of $X^2$. Therefore $C(X,2)$, which agrees with $C_{X^2,F_2}$, can be used in Theorem~\ref{maintheorem} to get a strong and $\mathcal{S}_2$-equivariant deformation retract of $\F(|X|,2)$. The only drawback is that $C(X,2)$ may fail to be $\mathcal{S}_2$-regular (see Example~\ref{concreto} below). Thus, if we care for a simplicial model for the unordered configuration space $\UF(|X|,2)$, we would still have to take a suitable subdivision of $C(X,2)$. The barycentric subdivision is a good option: not only is $bs(C(X,2))$ $\mathcal{S}_2$-regular (for
$C(X,2)$ is semiregular, as shown in Lemma~\ref{ulti}) but, as illustrated in Example~\ref{concreto} below, the $\Sigma_2$-quotient of $bs(C(X,2))$ is smaller than $UC_{bs}(X,2)$.
}\end{remark}

\red{The authors wish to thank Professor Wiltshire-Gordon for a number of insightful comments on a preliminary version of this work. In fact, his kind and illuminating feedback clearified our ideas, which led us to formulate and prove the equivariant property of his simplicial difference.}

\section{$G$-equivariant $C_{X,A}$ model}\label{secM}
In this short section we observe that~\cite[Lemma~70.1]{MR755006} extends in a straightforward way to the equivariant case. Let $A$ be a full subcomplex of an abstract simplicial complex~$X$, and let $C=C_{X,A}$ stand for the subcomplex of $X$ consisting of the simplices of $X$ whose geometric realization is disjoint from $|A|$. As shown in \cite[Lemma~70.1]{MR755006}, a homotopy $$H\colon\left(\rule{0mm}{4mm}|X|-|A|\right)\times [0,1]\to |X|-|A|$$ exhibiting $C$ as a strong deformation retract of $|X|-|A|$ is defined by the formula $H(x,s)=(1-s)\cdot x+s\cdot f(x)$, where
\begin{itemize}
\item $x=\sum_{i=1}^rt_ic_i+\sum_{j=1}^\rho\tau_ja_j$;
\item $c_1,\ldots c_r$ are vertices of $C$ ($r\geq1$);
\item $a_1,\ldots,a_\rho$ are vertices of $A$ ($\rho\geq0$);
\item $\sum_it_i+\sum_j\tau_j=1$ and $t_i>0<\tau_j$ for all $i$ and $j$;
\item $f(x)=\sum_{i=1}^r\frac{\raisebox{.3mm}{\scriptsize$t_i$}}{\lambda}\rule{.3mm}{0mm}c_i$, where $\lambda=\sum_{k=1}^rt_k$.
\end{itemize}

\begin{proposition}\label{hacerloequivariante}
In the setting above, assume a group $G$ acts simplicially on $X$ so that the action restricts to one on $A$. We have:
\begin{itemize}
\item[(1)] The $G$-action restricts to one on $C$.
\item[(2)] The homotopy $H$ is $G$-equivariant. Thus $|C|$ sits inside $|X|-|A|$ as a $G$-equivariant strong deformation retract.
\item[(3)] If the $G$-action on $X$ is regular (semiregular), then so is its restriction to $A$ and $C$.
\end{itemize}
\end{proposition}

\begin{remark}{\em
As noted in Munkres' book, $A$ and $C$ play symmetric roles in Proposition~\ref{hacerloequivariante}, so we also have that $|A|$ sits inside $|X|-|C|$ as a $G$-equivariant strong deformation retract.
}\end{remark}

\begin{remark}[{cf.~\cite[Section~III.1]{MR0413144}}]\label{definicionderegularidad}{\em A simplicial action of a group $G$ on an abstract simplicial complex $X$ is said to be:
\begin{itemize}
\item \emph{semiregular} if no two vertices of $X$ in a common $G$-orbit span a 1-dimensional simplex in $X$;
\item \emph{regular} if, whenever there are two simplices $\{v_0,\ldots, v_d\}$ and $\{g_0v_0,\ldots, g_dv_d\}$ of $X$, there should exist $g\in G$ with $gv_i=g_iv_i$ for $i=0,\ldots,d$ ---so the two simplices lie on the same $G$-orbit.
\end{itemize}
}\end{remark}

\begin{proof}[Proof of Proposition~\ref{hacerloequivariante}]
The first and third assertions are obvious, while the second one follows directly from the definition of $H$, observing that the retraction $f=H(-,1)\colon|X|-|A|\to |C|$ is $G$-equivariant.
\end{proof}

\section{Proof of Theorem~\ref{maintheorem}}\label{secT}
The first assertion in Thereom~\ref{maintheorem} follows from Proposition~\ref{hacerloequivariante}(2) and the discussion in the introductory section. For the second assertion, we only need to check the regularity of the $\mathcal{S}_n$-action on $C_{bs}(X,n)$. In turn, the asserted regularity is a consequence of~\cite[Proposition~III.1.1]{MR0413144}, Proposition~\ref{hacerloequivariante}(3), and Lemma~\ref{ulti} below.

Recall that the order-product simplicial structure in $X^n$ is taken with respect to a fixed linear order on the vertices of $X$. Thus, a set of $k$ vertices
\begin{equation}\label{productsimplex}
\{(v_{1,1},v_{2,1},\ldots,v_{n,1}),\ldots,(v_{1,k},v_{2,k},\ldots,v_{n,k})\}
\end{equation}
in $X^n$ determines a $(k-1)$-dimensional simplex whenever, for all $i\in\{1,\ldots,n\}$, the sequence $v_{i,1}, v_{i,2},\ldots, v_{i,k}$ is an \emph{ordered} simplex in $X^n$, i.e., the sequence of inequalities $v_{i,1}\leq v_{i,2}\leq\cdots\leq v_{i,k}$ holds, and $\{v_{i,1}, v_{i,2},\ldots, v_{i,k}\}$ is a simplex in $X$. As in~\cite{MR3887191}, it is convenient to think of~(\ref{productsimplex}) as the $n\times k$ matrix $(v_{i,j})$. Thus, rows of this matrix \emph{(i)} determine a non-decreasing sequence of vertices in $X$, and \emph{(ii)} span a simplex of $X$. The matrix has no repeated columns, since~(\ref{productsimplex}) has cardinality $k$.

\begin{lemma}\label{ulti}
$X^n$ is $\mathcal{S}_n$-semiregular, i.e., no two vertices in a common $\mathcal{S}_n$-orbit span a 1-dimensional simplex in $X^n$. In particular, $C(X,n)$ is $\mathcal{S}_n$-semiregular.
\end{lemma}
\begin{proof}
To see the semiregularity of $X^n$, let $\tau\in\mathcal{S}_n$ and assume, for a contradiction, that $v=(v_1,v_2,\ldots,v_n)$ is a vertex of $X^n$ which, together with $\tau\cdot v=(v_{\tau(1)},v_{\tau(2)},\ldots,v_{\tau(n)})$, generates an edge of~$X^n$ (so $v\neq\tau\cdot v$). Without loss of generality, we can assume $v_i\leq v_{\tau(i)}$ for all $i$. The largest $v_{i_1}$ must then satisfy $\tau(i_1)=i_1$. In turn, the next-to-the-largest $v_{i_2}$ is then forced to satisfy $\tau(i_2)=i_2$. Continuing this way, we see that $\tau$ would have to be the identity, a contradiction.

To see the semiregularity of $C(X,n)$, take $\tau\in\mathcal{S}_n$ a non-trivial permutation, say $\tau(i_0)\neq i_0$ ($1\leq i_0\leq n$). Rows $i_0$ and $\tau(i_0)$ of any vertex $(v_{i,j})$ of $C(X,n)$ (i.e., any simplex of $X^n$ that is minimal among non-simplices of $F_n$) must be different, so that
\begin{equation}\label{paralacontr}
v_{i_0,j_0}\neq v_{\tau(i_0),j_0}
\end{equation}
for some $j_0\in\{1,2,\ldots,n\}$. If $(v_{i,j})$ and $\tau\cdot(v_{i,j})$ were to span an edge of $C(X,n)$, i.e., if the columns of these matrices could be assembled into a larger matrix giving a simplex of $X^n$, then the $j_0$-th column $v_{j_0}:=(v_{1,j_0},\ldots,v_{n,j_0})$ of $(v_{i,j})$ and the $j_0$-th column $\tau\cdot v_{j_0}$ of $\tau\cdot(v_{i,j})$ would be vertices of $X^n$ generating a simplex of $X^n$. From the first half of the proof, this could hold only with $v_j=\tau\cdot v_j$, which is impossible in view of~(\ref{paralacontr}).
\end{proof}

\begin{example}\label{concreto}{\em 
Let $X=\partial\Delta^2$. The (the geometric realization of) $X^2$ can be depicted as
$$
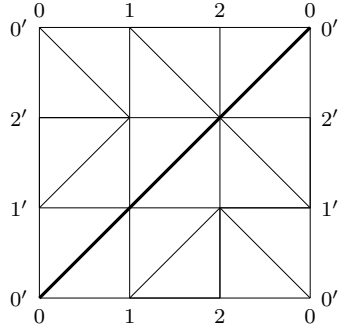
\begin{tikzpicture}[x=.6cm,y=.6cm]
\node [left] at (0,0) {\scriptsize $0'$};
\node [left] at (0,2) {\scriptsize $1'$};
\node [left] at (0,4) {\scriptsize $2'$};
\node [left] at (0,6) {\scriptsize $0'$};
\node [right] at (6,0) {\scriptsize $0'$};
\node [right] at (6,2) {\scriptsize $1'$};
\node [right] at (6,4) {\scriptsize $2'$};
\node [right] at (6,6) {\scriptsize $0'$};
\node [below] at (0,0) {\scriptsize $0$};
\node [below] at (2,0) {\scriptsize $1$};
\node [below] at (4,0) {\scriptsize $2$};
\node [below] at (6,0) {\scriptsize $0$};
\node [above] at (0,6) {\scriptsize $0$};
\node [above] at (2,6) {\scriptsize $1$};
\node [above] at (4,6) {\scriptsize $2$};
\node [above] at (6,6) {\scriptsize $0$};
\draw[very thin](2,0)--(4,0); \draw[very thin](0,0)--(6,0);
\draw[very thin](2,6)--(4,6); \draw[very thin](0,6)--(6,6); 
\draw[very thin](0,2)--(0,4); \draw[very thin](0,0)--(0,6);
\draw[very thin](6,2)--(6,4); \draw[very thin](6,0)--(6,6);
\draw[very thick](0,0)--(6,6);
\draw[very thin](6,2)--(2,6);
\draw[very thin](4,2)--(6,2); \draw[very thin](0,2)--(6,2);
\draw[very thin](0,4)--(2,4); \draw[very thin](0,4)--(6,4);
\draw[very thin](2,4)--(2,6); \draw[very thin](2,0)--(2,4);
\draw[very thin](4,0)--(4,2); \draw[very thin](4,0)--(4,6);
\draw[very thin](2,0)--(4,2);
\draw[very thin](4,2)--(6,0);
\draw[very thin](0,2)--(2,4);
\draw[very thin](2,4)--(0,6);
\end{tikzpicture}$$
where opposite sides of the boundary square are identified as indicated. The $\mathcal{S}_2$-action is given by reflection on the diagonal (full, as noted in Remark~\ref{recomendacion}) subcomplex $F_2$ ---the latter is hightlighted with thick segments. The (geometric realization of the) subcomplex $C(X,2)=C_{X^2,F_2}$, highlighted on thick lines in the picture
$$
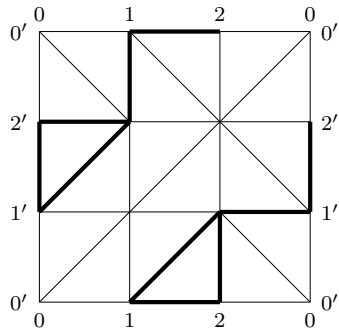
\begin{tikzpicture}[x=.6cm,y=.6cm]
\node [left] at (0,0) {\scriptsize $0'$};
\node [left] at (0,2) {\scriptsize $1'$};
\node [left] at (0,4) {\scriptsize $2'$};
\node [left] at (0,6) {\scriptsize $0'$};
\node [right] at (6,0) {\scriptsize $0'$};
\node [right] at (6,2) {\scriptsize $1'$};
\node [right] at (6,4) {\scriptsize $2'$};
\node [right] at (6,6) {\scriptsize $0'$};
\node [below] at (0,0) {\scriptsize $0$};
\node [below] at (2,0) {\scriptsize $1$};
\node [below] at (4,0) {\scriptsize $2$};
\node [below] at (6,0) {\scriptsize $0$};
\node [above] at (0,6) {\scriptsize $0$};
\node [above] at (2,6) {\scriptsize $1$};
\node [above] at (4,6) {\scriptsize $2$};
\node [above] at (6,6) {\scriptsize $0$};
\draw[ultra thick](2,0)--(4,0); \draw[ultra thin](0,0)--(6,0);
\draw[ultra thick](2,6)--(4,6); \draw[ultra thin](0,6)--(6,6); 
\draw[ultra thick](0,2)--(0,4); \draw[ultra thin](0,0)--(0,6);
\draw[ultra thick](6,2)--(6,4); \draw[ultra thin](6,0)--(6,6);
\draw[ultra thin](0,0)--(6,6);
\draw[ultra thin](6,2)--(2,6);
\draw[ultra thick](4,2)--(6,2); \draw[ultra thin](0,2)--(6,2);
\draw[ultra thick](0,4)--(2,4); \draw[ultra thin](0,4)--(6,4);
\draw[ultra thick](2,4)--(2,6); \draw[ultra thin](2,0)--(2,4);
\draw[ultra thick](4,0)--(4,2); \draw[ultra thin](4,0)--(4,6);
\draw[ultra thick](2,0)--(4,2);
\draw[ultra thin](4,2)--(6,0);
\draw[ultra thick](0,2)--(2,4);
\draw[ultra thin](2,4)--(0,6);
\end{tikzpicture}$$
is not $\mathcal{S}_2$-regular. For instance, if $\tau$ stands for the generator of $\mathcal{S}_2$ (so $\tau^2=e$, the neutral element), then the edges $\{(1,0'),(2,1')\}$ and $\{e\cdot(1,0'),\tau\cdot(2,1')\}$ of $C(X,2)$ do not form a $\mathcal{S}_2$-orbit. Nevertheless, as indicated in Remark~\ref{recomendacion}, the (geometric realization of the) simplicial complex $bs(C(X,2))/\mathcal{S}_2$ gives a simplicial up-to-homotopy model for the unordered configuration space $\UF(|X|,2)$. With $8$~vertices, $14$~edges, and $6$~faces, $bs(C(X,2))/\mathcal{S}_2$ is much smaller than $UC_{bs}(X,2)$, which has $25$~vertices, $61$~edges, and $36$~faces, as can be seen from Figure~\ref{2fgs}.
\begin{figure}
\begin{center}
\begin{tikzpicture}[x=.5cm,y=.5cm]
\draw[very thick](0,0)--(6,0); \draw[very thick](0,0)--(0,6);
\draw[very thick](0,6)--(6,6); \draw[very thick](6,0)--(6,6);
\draw[very thick](0,0)--(6,6); \draw[very thick](6,2)--(2,6); 
\draw[very thick](0,2)--(6,2); \draw[very thick](0,4)--(6,4); 
\draw[very thick](2,0)--(2,6); \draw[very thick](4,0)--(4,6);
\draw[very thick](2,0)--(4,2); \draw[very thick](4,2)--(6,0);
\draw[very thick](0,2)--(2,4); \draw[very thick](2,4)--(0,6);
\draw[ultra thin](4,2)--(6,4); \draw[ultra thin](6,2)--(4,3);
\draw[ultra thin](4,4)--(6,0); \draw[ultra thin](4,4)--(6,3);
\draw[ultra thin](6,2)--(5,4); \draw[ultra thin](4,0)--(6,2);
\draw[ultra thin](4,2)--(6,1); \draw[ultra thin](4,2)--(5,0);
\draw[ultra thin](4,1)--(6,0); \draw[ultra thin](1,4)--(0,6);
\draw[ultra thin](2,4)--(4,6); \draw[ultra thin](2,6)--(3,4);
\draw[ultra thin](4,4)--(0,6); \draw[ultra thin](4,4)--(3,6);
\draw[ultra thin](2,6)--(4,5); \draw[ultra thin](0,4)--(2,6);
\draw[ultra thin](2,4)--(1,6); \draw[ultra thin](2,4)--(0,5);
\draw[ultra thin](4,6)--(6,4); \draw[ultra thin](4,4)--(5,6);
\draw[ultra thin](6,6)--(4,5); \draw[ultra thin](6,6)--(5,4);
\draw[ultra thin](4,4)--(6,5); \draw[ultra thin](0,4)--(4,0);
\draw[ultra thin](0,2)--(2,0); \draw[ultra thin](2,4)--(4,2);
\draw[ultra thin](4,4)--(2,0); \draw[ultra thin](3,4)--(1,0);
\draw[ultra thin](2,4)--(0,0); \draw[ultra thin](1,4)--(0,2);
\draw[ultra thin](2,4)--(0,3); \draw[ultra thin](4,4)--(0,2);
\draw[ultra thin](4,2)--(0,0); \draw[ultra thin](4,3)--(0,1);
\draw[ultra thin](2,0)--(4,1); \draw[ultra thin](4,2)--(3,0);
\draw[very thick](10,1)--(10,6);
\draw[very thick](10,6)--(15,6);
\draw[very thick](13,5)--(12,6); 
\draw[very thick](10,2)--(11,2); \draw[very thick](10,4)--(13,4);
\draw[very thick](12,3)--(12,6); \draw[very thick](14,5)--(14,6);
\draw[very thick](10,2)--(12,4); \draw[very thick](12,4)--(10,6);
\draw[ultra thin](11,4)--(10,6);
\draw[ultra thin](12,4)--(14,6); \draw[ultra thin](12,6)--(13,4);
\draw[ultra thin](12.66,4.67)--(10,6); 
\draw[ultra thin](13.334,5.33)--(13,6);
\draw[ultra thin](12,6)--(14,5); \draw[ultra thin](10,4)--(12,6);
\draw[ultra thin](12,4)--(11,6); \draw[ultra thin](12,4)--(10,5);
\draw[ultra thin](14,6)--(10+14/3,16/3);
\draw[ultra thin](15,6)--(10+14/3,16/3);
\draw[ultra thin](14,5)--(10+14/3,16/3);
\draw[ultra thin](11,4)--(10,2);
\draw[ultra thin](12,4)--(10,3);
\draw[ultra thin](12,4)--(10+2/3,4/3);
\draw[ultra thin](10,2)--(10+2/3,4/3);
\draw[ultra thin](10,1)--(10+2/3,4/3);
\draw[ultra thin](10,4)--(10+4/3,2+2/3);
\draw[ultra thin](10,2)--(12,3);
\draw[ultra thin](12,4)--(38/3,16-38/3);
\draw[ultra thin](12,3)--(38/3,16-38/3);
\draw[ultra thin](13,4)--(38/3,16-38/3);
\node [above] at (11,6) {\scriptsize $a$};
\node [above] at (12,6) {\scriptsize $b$};
\node [above] at (13,6) {\scriptsize $c$};
\node [above] at (14,6) {\scriptsize $d$};
\node [above] at (15,6) {\scriptsize $e$};
\node [left] at (10,1) {\scriptsize $a$};
\node [left] at (10,2) {\scriptsize $b$};
\node [left] at (10,3) {\scriptsize $c$};
\node [left] at (10,4) {\scriptsize $d$};
\node [left] at (10,5) {\scriptsize $e$};
\node [left] at (0,0) {\scriptsize $\alpha$};
\node [left] at (0,1) {\scriptsize $\beta$};
\node [right] at (6,1) {\scriptsize $\beta$};
\node [left] at (0,2) {\scriptsize $\gamma$};
\node [right] at (6,2) {\scriptsize $\gamma$};
\node [left] at (0,3) {\scriptsize $\delta$};
\node [right] at (6,3) {\scriptsize $\delta$};
\node [left] at (0,4) {\scriptsize $\epsilon$};
\node [right] at (6,4) {\scriptsize $\epsilon$};
\node [left] at (0,5) {\scriptsize $\zeta$};
\node [right] at (6,5) {\scriptsize $\zeta$};
\node [left] at (0,6) {\scriptsize $\alpha$};
\node [right] at (6,0) {\scriptsize $\alpha$};
\node [right] at (6,6) {\scriptsize $\alpha$};
\node [below] at (1,0) {\scriptsize $\eta$};
\node [above] at (1,6) {\scriptsize $\eta$};
\node [below] at (2,0) {\scriptsize $\theta$};
\node [above] at (2,6) {\scriptsize $\theta$};
\node [below] at (3,0) {\scriptsize $\iota$};
\node [above] at (3,6) {\scriptsize $\iota$};
\node [below] at (4,0) {\scriptsize $\kappa$};
\node [above] at (4,6) {\scriptsize $\kappa$};
\node [below] at (5,0) {\scriptsize $\lambda$};
\node [above] at (5,6) {\scriptsize $\lambda$};
\end{tikzpicture}
\end{center}
\caption{The torus $bs(X^2)$ (left) and the M\"obius band $UC_{bs}(X,2)$ (right)}
\label{2fgs}
\end{figure}
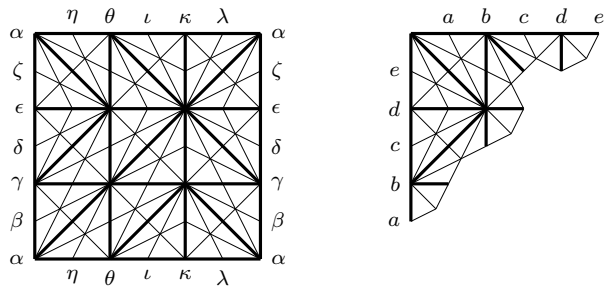
}\end{example}

\begin{remark}\label{mejor}{\em
Dimension might be a reason to prefer $C_{bs}(X,n)$ over $bs(C(X,n))$ or even $C(X,n)$. The latter complex can turn out to have dimension larger than that of $\F(|X|,n)$, whereas $\dim(C_{bs}(X,n))\leq n\dim(X)$ holds in general. As a representative instance of this phenomenon, take the one dimensional complex $X=\partial\Delta^2$ with vertices ordered as $1,2,3$. Then the set of vertices
$$
\left\{
\begin{pmatrix}1&1\\1&2\\2&2\end{pmatrix},
\begin{pmatrix}1&3\\1&2\\2&2\end{pmatrix},
\begin{pmatrix}1&3\\1&2\\2&3\end{pmatrix},
\begin{pmatrix}1&3\\2&2\\2&3\end{pmatrix},
\begin{pmatrix}3&3\\2&2\\2&3\end{pmatrix}
\right\}
$$
is easily seen to give a 4-dimensional simplex in $C(X,3)$. 
}\end{remark}

\section{Equivariant simplicial difference}\label{secW}
Let $B$ be a (not necessarily full) subcomplex of the abstract simplicial complex $Y$. Wiltshire-Gordon obtains a homotopy equivalence
\begin{equation}\label{ebnl}
|Y|-|B|\simeq |Y\ominus B|
\end{equation}
as an application of the Nerve Lemma. Indeed, for a simplex $\sigma$ of $Y$, consider the open subset $U_\sigma$ of $|Y|$ consisting of barycentric expressions $\sum_{v\in Y^{(0)}} t_v \cdot v$ satisfying $\sigma\subseteq\{v\in Y^{(0)}\colon t_v>0\}$. Note that $U_\sigma$ is star-shaped (thus contractible) with respect to the barycenter $b_\sigma$ of $|\sigma|$. Furthermore, by definition, a finite intersection $\bigcap_i U_{\sigma_i}$ is nonempty if and only if $\bigcup_i\sigma_i\in Y$, in which case
\begin{equation}\label{iuequalsuu}
\bigcap_i U_{\sigma_i}=U_{\bigcup_i\sigma_i}.
\end{equation}
Since $|Y|-|B|$ is covered by the family $$\mathcal{U}:=\{U_\sigma\colon \sigma \mbox{ is a minimal non-face of } B\},$$ whose nerve is $Y\ominus B$ (as defined in the introductory section), the Nerve Lemma yields a homotopy equivalence~(\ref{ebnl}). In particular, by taking  the ``fat diagonal" subcompex $B:=F_n$ of $Y:=X^n$ we get Wiltshire-Gordon's homotopy equivalence
\begin{equation}\label{hhdhd}
|C(X,n)|\simeq\F(|X|,n).
\end{equation}

As noted in the introductory section, the topological $\mathcal{S}_n$-action on $\F(|X|,n)$ yields, in a very natural way, a  simplicial $\mathcal{S}_n$-action on $C(X,n)$. However, this does not mean \emph{a priori} that the homotopy equivalence~(\ref{hhdhd}) ---coming from the Nerve Lemma--- would have to be equivariant. % (as \red{asserted} in~\cite[Remark~1.5]{MR3887191}).
\red{What is needed here is a suitable equivariant version of the Nerve Lemma assuring an $\mathcal{S}_n$-equivariant equivalence~(\ref{hhdhd}). More generally, in this section we use an equivariant form of the Nerve Lemma (to be proved in the next section) in order to obtain a $G$-equivariant homotopy equivalence~(\ref{ebnl}) provided a finite group $G$ acts simplicially on the pair $(Y,B)$. In preparation for the subtleties of such a result, a few notational details are discussed next.}

\begin{remark}\label{casogeneralito}{\em
Let $\mathcal{V}=\{V_i\}_{i\in I}$ be a good open cover of a space $X$, i.e., an open cover satisfying the hypothesis in the Nerve Lemma that each finite nonempty intersection $\cap_jV_j$ is contractible. In one of the standard proofs of the Nerve Lemma (see for instance~\cite[Section~4.G]{MR1867354} or~\cite[Theorem~15.21]{MR2361455}), the homotopy types of $X$ and $|N(\mathcal{V})|$ are identified by constructing explicit homotopy equivalences
\begin{equation}\label{laZ}
|N(\mathcal{V})|\leftarrow H\rightarrow X
\end{equation}
defined on a certain common (homotopy colimit) space $H$. Constructions are natural enough to expect \emph{equivariant} homotopy equivalences in~(\ref{laZ}) \emph{provided} a group $G$ acts on $X$ in a way that is compatible with the hypotheses on the Nerve Lemma. To begin with, $G$ should act on the cover $\mathcal{V}$, i.e., on its index set $I$, so to guaranty a corresponding simplicial action of $G$ on the nerve $N(\mathcal{V})$ ---the latter one is used in the equivariant conclusion. This means that $g\cdot V_i=V_{g\cdot i}$ for all $i\in I$ and $g\in G$, a condition that is expressed as the cover being \emph{$G$-invariant}. But more importantly, in order to expect a $G$-equivariant homotopy equivalence $|N(\mathcal{V})|\simeq_G X$, the $G$-action and the contracting homotopies on each nonempty finite intersection $\cap_jV_j$ \red{should be compatible.}
}\end{remark}

Compatibility requirements as those noted at the end of Remark~\ref{casogeneralito} have been spelled out in the literature in a number of different ways, depending on the applications in mind. The following statement seems to be the strongest form (regarding both hypothesis and conclusion) in print of an Equivariant Nerve Lemma:

\begin{theorem}[{\cite[Theorem~2.19]{MR3215956}}]\label{yang}
Let $G$ be a finite group and $\mathcal{V}=\{V_i\}_{i\in I}$ be a locally finite $G$-invariant and $G$-equivariant good cover of a $G$-CW complex $X$. Then there is a $G$-equivariant homotopy equivalence $|N(\mathcal{V})|\simeq_G X$.
\end{theorem}

The ``$G$-equivariant good cover'' hypothesis in Theorem~\ref{yang} contains the requirement that the nerve $N(\mathcal{V})$ be $G$-regular (see top of page 231 and Definition 2.7 in~\cite{MR3215956}). So, implicit in Theorem~\ref{yang} is the assumption that, whenever a finite collection of open sets $U_{i_j}$ has a nonempty intersection, and there are elements $g_j\in G$ such that $\bigcap_{j}g_j \cdot V_{i_j}\neq\varnothing$, there must be some element $g\in G$ such that $g\cdot V_{i_j}=g_j\cdot V_{i_j}$ for every relevant $j$ (see Remark~\ref{definicionderegularidad}). This (and the equivariant Whitehead theorem) allows Yang to assemble the several contracting homotopies into the required $G$-equivariant homotopy equivalence in the conclusion of Theorem~\ref{yang}.

Unfortunatelly, Theorem~\ref{yang} cannot be used to assure an $\mathcal{S}_n$-equivariant homotopy equivalence~(\ref{hhdhd}). Indeed, as illustrated in Example~\ref{concreto}, the regularity condition on $C(X,n)$ fails in general.

Besides Yang's, the only other Equivariant Nerve Lemma that the authors are aware of is the following statement, which is better suited for our goals, as it imposes a less restrictive (regularity-wise) hypothesis:

\begin{theorem}[{\cite[Lemma~2.5]{MR3003329}}]\label{HH}
Let $G$ be a group and $\Delta$ be a $G$-simplicial complex with a $G$-invariant covering $\mathcal{D}=\{\Delta_i\}_{i\in I}$ by subcomplexes. Assume that every nonempty finite intersection $\bigcap_{i\in\sigma}\Delta_i$, where $\sigma\subseteq I$, is $G_\sigma$-contractible. Then there is a $G$-equivariant homotopy equivalence $|N(\mathcal{D})|\simeq_G \Delta$.
\end{theorem}

Here $G_\sigma$ is the isotropy subgroup of $G$ with respect to the simplex $\sigma$ of the nerve $N(\mathcal{D})$, so that the $G$-action on $X$ restrict to one of $G_\sigma$ on $\bigcap_{i\in\sigma}\Delta_i$.

\begin{remark}\label{semreg}{\em
Let us stress on the fact that regularity of the nerve is not required in Hess-Hirsch's Equivariant Nerve Lemma. Instead, the needed compatibility between the $G$-action and the hypothesis of the Nerve Lemma is packed into the requirement that the contracting homotopies are equivariant with respect to the corresponding isotropy subgroups. Note that the latter compatibility requirement is actually weaker than requiring that the nerve is $G$-semiregular: By definition (Remark~\ref{definicionderegularidad}), if the nerve is $G$-semiregular, then each isotropy subgroup $G_\sigma$ must be trivial.
}\end{remark}

\begin{remark}\label{prevented}{\em
By Lemma~\ref{ulti} and the last assertion in Remark~\ref{semreg}, the compatibility condition in Theorem~\ref{HH} holds for the cover $\{U_\sigma\}_{\sigma\in X^n\ominus F_n}$ of $\F(|X|,n)$. However, we are prevented from using this version of the Equivariant Nerve Lemma (to get a $\mathcal{S}_n$-equivariant homotopy equivalence~(\ref{hhdhd})) because Hess-Hirsch's statement is set entirely in the simplicial category. Not only they assume that the space is a simplicial complex, which is not (at least directly) the case for $\F(|X|,n)$), but they require that the cover be by subcomplexes (rather than by open sets). Needless to say, the last point is relevant in Hess-Hirsch's simplicial method of proof.
}\end{remark}

We overcome the drawback noted in Remark~\ref{prevented} by means of the following topological version of Theorem~\ref{HH}, which is proved in the next section along the ideas in Remark~\ref{casogeneralito}.

\begin{theorem}[Topological version of Hess-Hirsch's Equivariant Nerve Lemma]\label{conejura}
Let $G$ be a finite group and $X$ be a paracompact $G$-space with a locally finite $G$-invariant open covering $\hspace{.6mm}\mathcal{U}=\{U_i\}_{i\in I}$. If the intersection $\bigcap_{i\in\sigma}U_i$ is $G_\sigma$-contractible for every simplex $\sigma$ of $N(\mathcal{U})$, then there is a $G$-equivariant homotopy equivalence $|N(\mathcal{U})|\simeq_G X.$
\end{theorem}

Our interest in Theorem~\ref{conejura} stems from the fact that it yields the equivariant homotopy equivalence in Theorem~\ref{equiW} below ---a partial (but much more elaborate) analogue of Proposition~\ref{hacerloequivariante}.

The next result (and its proof) uses the notation and constructions in~\cite{MR3887191} reviewed at the beginning of the section.

\begin{lemma}\label{casicasi}
Let $X$ be an abstract simplicial complex with a (simplicial) action of a group $G$. If the $G$-action restricts to one on a subcomplex $A$, then:
\begin{enumerate}
\item The topological $G$-action on $|X|$ restricts to one on $|X|-|A|$.
\item The covering $\{U_\sigma\hspace{.7mm} \colon \sigma\text{ is a vertex of }X\ominus A \}$ is $G$-invariant (with $U_\sigma=U_\tau$ holding only for $\sigma=\tau$).
\item For a simplex $\{\sigma_0,\ldots,\sigma_d\}$ of $X\ominus A$, $G_{\{\sigma_0,\ldots,\sigma_d\}}\subseteq G_{\sigma_0\cup\cdots\cup\sigma_d}$. Note that the former (respectively, latter) isotropy subgroup is taken with respect to the $G$-action on $X\ominus A$ (respectively, on $X$).
\item For any simplex $\sigma$ of $X$, $U_\sigma$ is $G_\sigma$-contractible.
\end{enumerate}
\end{lemma}
\begin{proof}
Parts 1 and 3 are obvious. For part 2, just observe that $g\cdot U_\sigma=U_{g\cdot\sigma}$ and that, if $\tau\not\subseteq\sigma$, then the barycenter $b_\sigma$ of $|\sigma|$ lies in $U_\sigma-U_\tau$. Lastly, part 4 holds since the contracting homotopy for $U_\sigma$ is linear, $H(x,t)=(1-t)x+t b_\sigma$, so that, for $g\in G_\sigma$ and $x\in U_\sigma=U_{g\cdot\sigma}$, we have $b_\sigma=b_{g\cdot\sigma}=:b$ and thus $H(gx,t)=(1-t)gx+tb=(1-t)gx+tgb=g\cdot((1-t)x+tb)=g\cdot H(x,t)$.
\end{proof}

\begin{theorem}\label{equiW}
Let $X$, $A$ and $G$ be as in Lemma~\ref{casicasi}. If $G$ is finite, there is a homotopy equivalence~(\ref{ebnl}) that is $G$-equivariant.
\end{theorem}
\begin{proof}
Use~(\ref{iuequalsuu}), Lemma~\ref{casicasi} and Theorem~\ref{conejura}.
\end{proof}

As noted in Remark~\ref{softened} in the next section, the finiteness hypothesis on $G$ of Theorem~\ref{conejura} (and thus of Theorem~\ref{equiW}) can be softened to requiring that $G$ is compact.

\begin{remark}\label{rfl}{\em
\red{In a preliminary version of this work we suggested that a suitable Equivariant Nerve Lemma filling up the equivariant details in~\cite{MR3887191} would be provable on the lines of Remark~\ref{casogeneralito}. This belief was supported when the authors learned from Professor Wiltshire-Gordon that the standard proof of the Nerve Lemma can indeed be adapted to prove a particular case of Theorem~\ref{conejura} which, in essence, amounts to assuming the stronger hypothesis that the nerve of the cover is semiregular (see the last assertion in Remark~\ref{semreg}). As noted by Wiltshire-Gordon~\cite{private}, this leads to a validation of~(\ref{desliz}) (c.f.~the first assertion in Remark~\ref{prevented}). An alternative shortcut argument validating~(\ref{desliz}) can be obtained from deeper properties in equivariant homotopy theory (see Remark~\ref{deeper} in the next section). Yet, the more general Theorem~\ref{equiW} depends heavily on the full form of our Equivariant Nerve Lemma in Theorem~\ref{conejura}.}
}\end{remark}

The Equivariant-Nerve-Lemma viewpoint was our first approach to \red{filling up} the equivariant issue in~\cite{MR3887191}. Our reason for preferring the solution in terms of Proposition~\ref{hacerloequivariante} is two-folded. For one, the latter result is conceptually much simpler than the one coming from the Equivariant Nerve Lemma. Furthermore, and more importantly, the conclusion thus obtained is stronger, as Theorem~\ref{maintheorem} yields a $\mathcal{S}_n$-regular model whose realization sits inside $\F(|X|,n)$ as a strong and equivariant deformation retract---a condition that is desirable for applications, but that would not be drawable from the Equivariant Nerve Lemma. Yet, the Equivariant Nerve Lemma formulated in Theorem~\ref{conejura} is interesting on its own, and has potential applications in equivariant guises of computational and combinatorial topology.

\section{Equivariant Nerve Lemma}
In this section we adapt the proof of the Nerve Lemma, as described in~\cite[Corollary~4G.3]{MR1867354} and~\cite[Theorem~15.21]{MR2361455}, to the equivariant situation. We start by developing the basic definitions and constructions in the equivariant context.

Let $B$ be a $\Delta$-complex, as described in~\cite[Section~2.1]{MR1867354} (the underlying combinatorial structure is called a trisp in~\cite[Chapter~15]{MR2361455}). Thus, $B$ is a cell complex with cells $\sigma$ (also written as $\sigma^d$ to stress their dimension) and attaching maps $\phi_\sigma\colon \Delta^d\to B$, defined on \emph{linearly ordered} simplices, that are face-coherent with respect to order-preserving affine inclusions. In particular, edges in~$B$ are canonically oriented. In this paper we will only consider the case when $B$ a regular $\Delta$-complex~(\cite[Definition~2.47]{MR2361455}), which means that all attaching maps are embeddings. In particular our $\Delta$-complexes are in fact simplicial, with the vertices of each simplex inheriting a canonical linear order. We will thus write $[v_0,v_1,\ldots,v_d]$ when referring to the $d$-simplex of $B$ with vertices $v_i$ satisfying $v_0<v_1<\cdots<v_d$.

Let $G$ be a group. We say that a $\Delta$-complex $B$ is a $G\text{-}\Delta$-complex if $G$ acts topologically on $B$, as well as on each set of $d$-dimensional simplices $S_d$, in such a way that each composite
\begin{equation}\label{gdeltacx}
\Delta^d\stackrel{\phi_\sigma}{\longrightarrow}B\stackrel{g}{\longrightarrow}B
\end{equation}
($g\in G$ and $\sigma\in S_d$) agrees with $\phi_{g\cdot\sigma}$, where the second map in~(\ref{gdeltacx}) is multiplication by $g$. In particular, given our regularity assumption, $G$ acts simplicially and order-preserving on $B$. In terms of barycentric coordinates, the latter condition means that $g\sum_i{t_i v_i}=\sum_it_i \,gv_i$, with $g v_i<g v_j$ whenever $v_i<v_j$.

Let $\mathcal{C}$ be a diagram of spaces over a $\Delta$-complex $B$ in the sense of~\cite[Definition~15.1]{MR2361455}. As in~\cite[Section~4.G]{MR1867354}, we also use the name ``complex of spaces'' for such a structure. Thus $\mathcal{C}$ is a rule assigning a space $\mathcal{C}(v)$ to each vertex of $v$ of $B$, and a map $\mathcal{C}[v,u]\colon\mathcal{C}(v)\to \mathcal{C}(u)$ to each oriented edge $[v,u]$ of $B$, so that $\mathcal{C}[v,w]=\mathcal{C}[u,w]\circ\mathcal{C}[v,u]$ whenever $[v,u]$ and $[u,w]$ are edges of a common simplex of $B$. We say that $\mathcal{C}$ is a $G$-complex of spaces (or $G$-diagram of spaces) over the $G\text{-}\Delta$-complex $B$ if there are maps $\mathcal{C}_{g,v}\colon\mathcal{C}(v)\to\mathcal{C}(gv)$ (for every $g\in G$ and every vertex $v$ of~$B$) satisfying the following properties:
\begin{itemize}
\item Each $\mathcal{C}_{e,v}$ is the identity map, where $e$ stands for the neutral element of $G$.
\item For $g\in G$ and $v$ vertex of $B$, $\mathcal{C}_{h,gv}\circ\mathcal{C}_{g,v}=\mathcal{C}_{hg,v}$.
\item For $h,g\in G$ and $[v,u]$ oriented edge of $B$, $\mathcal{C}[gv,gu]\circ\mathcal{C}_{g,v}=\mathcal{C}_{g,u}\circ\mathcal{C}[v,u]$.
\end{itemize}
When the meaning is clear from the context, we will simplify the notation $\mathcal{C}_{g,v}\colon\mathcal{C}(v)\to\mathcal{C}(gv)$ to $g\colon\mathcal{C}(v)\to\mathcal{C}(gv)$. We will even suppress the label $\mathcal{C}[v,u]$ from a map $\mathcal{C}[v,u]\colon\mathcal{C}(v)\to\mathcal{C}(u)$. For instance, in such terms, the two final conditions in the definition of a $G$-complex of spaces translate into having commutative diagrams
$$\xymatrix{
\mathcal{C}(v) \ar[r]^{g} \ar[rd]_{hg} & \mathcal{C}(gv) \ar[d]^h &&
\mathcal{C}(v)\ar[r]^{g\;}\ar[d] & \mathcal{C}(gv)\ar[d]\\
& \mathcal{C}(hgv), && \mathcal{C}(u)\ar[r]^{g\;} & \mathcal{C}(gu).
}$$

Recall the colimit (colim) of a complex of spaces $\mathcal{C}$ over a $\Delta$-complex $B$, i.e., the quotient space
$$
\colim\mathcal{C}=\left(\bigsqcup_v\mathcal{C}(v)\right)\left.\rule{0mm}{5mm}\right/\sim
$$
where the disjoint union is taken over the vertices of $B$, and the equivalence relation is generated by $x\sim\mathcal{C}[v,u](x)$ for $x\in\mathcal{C}(v)$ and $[v,u]$ an oriented edge of $B$. The homotopy colimit (hocolim) of $\mathcal{C}$ is
$$
\hocolim\mathcal{C}=\left(\,\bigsqcup_{[v_0,\ldots,v_d]}[v_0,\ldots, v_d]\times\mathcal{C}(v_0)\right)\left.\rule{0mm}{5mm}\right/\sim
$$
where the disjoint union is taken over all (ordered) simplices of $B$, and the equivalence relation is generated by the following two types of identifications:
\begin{itemize}
\item $[v_0,\ldots,v_d]\times \mathcal{C}(v_0)\ni(\iota_i(\alpha),x)\sim(\alpha,x)\in [v_0,\ldots,\widehat{v_i},\ldots,v_d]\times\mathcal{C}(v_0)$, provided $i>0$.
\item $[v_0,\ldots,v_d]\times \mathcal{C}(v_0)\ni(\iota_0(\alpha),x)\sim(\alpha,\mathcal{C}[v_0,v_1](x))\in[v_1,\ldots,v_d]\times\mathcal{C}(v_1)$.
\end{itemize}
Here $\iota_j\colon [v_0,\ldots,\widehat{v_j},\ldots,v_d]\hookrightarrow [v_0,\ldots,v_d]$ stands for a general (order-preserving) face inclusion.

Let $\mathcal{C}_1$ and $\mathcal{C}_2$ be $G$-complexes of spaces over a $G\text{-}\Delta$-complex $B$, and let $\mathcal{F}\colon \mathcal{C}_1\to\mathcal{C}_2$ be a map of complexes of spaces in sense of~\cite[Definition~15.10]{MR2361455}. In other words, $\mathcal{F}$ is a collection of maps $\mathcal{F}(v)\colon\mathcal{C}_1(v)\to\mathcal{C}_2(v)$, one for each vertex $v$ of $B$, such that the diagram
$$\xymatrix{
\mathcal{C}_1(v)\ar[r]^{\mathcal{F}(v)} \ar[d]& \mathcal{C}_2(v)\ar[d]\\
 \mathcal{C}_1(u)\ar[r]^{\mathcal{F}(u)} & \mathcal{C}_2(u)
}$$
commutes for each oriented edge $[v,u]$ of $B$. We say that $\mathcal{F}\colon\mathcal{C}_1\to\mathcal{C}_2$ is a $G$-map if, for any $g\in G$ and any vertex $v$ of $B$, there is a commutative diagram
$$\xymatrix{
\mathcal{C}_1(v)\ar[r]^{\mathcal{F}(v)} \ar[d]_g& \mathcal{C}_2(v)\ar[d]^g\\
 \mathcal{C}_1(gv)\ar[r]^{\mathcal{F}(gv)} & \mathcal{C}_2(gv).
}$$

It follows from the definitions that, if $\mathcal{C}_1$ and $\mathcal{C}_2$ are $G$-complexes of spaces over a $G\text{-}\Delta$-complex $B$, and $\mathcal{F}\colon\mathcal{C}_1\to\mathcal{C}_2$ is a $G$-map, then $\colim\mathcal{C}_i$ and $\hocolim\mathcal{C}_i$ inherit obvious topological $G$-actions. Furthermore, the induced maps $\colim\mathcal{F}\colon\colim\mathcal{C}_1\to\colim\mathcal{C}_2$ and $\hocolim\mathcal{F}\colon\hocolim\mathcal{C}_1\to\hocolim\mathcal{C}_2$ are $G$-equivariant. Likewise, the structural maps $$p_f\colon\hocolim\mathcal{C}_1\to\colim\mathcal{C}_1 \text{ \ \ (fiber projection)}$$ and $$p_b\colon\hocolim\mathcal{C}_1\to\hocolim\widetilde{\mathcal{C}_1} \text{ \ \ (base projection)}$$ are $G$-equivariant too. Here $\widetilde{\mathcal{C}_1}$ is the obvious $G$-complex of spaces over $B$ such that each $\widetilde{\mathcal{C}_1}(v)$ is a one-point space. Note that the identification $\hocolim\widetilde{\mathcal{C}_1}=B$ is compatible with the $G$-actions.

\begin{example}{\em
Recall the complex of spaces $\mathcal{C}_{\mathcal{U}}$ associated to an open cover $\mathcal{U}=\{U_i\}_{i\in I}$ of a given space $X$ (\cite[Definition~15.15]{MR2361455}). Explicitly, $\mathcal{C}_{\mathcal{U}}$ is the complex of spaces over $bs(N\mathcal{U})$, the barycentric subdivision of the nerve $N\mathcal{U}$ of $\mathcal{U}$, with
\begin{itemize}
\item $\mathcal{C}_{\mathcal{U}}(\sigma):=\bigcap_{i\in \sigma}U_i$, for a vertex $\sigma$ of $bs(N\mathcal{U})$ (i.e., $\sigma$ a simplex of $N\mathcal{U}$).
\item $\mathcal{C}_{\mathcal{U}}[\sigma,\tau]\colon\mathcal{C}_{\mathcal{U}}(\sigma)\hookrightarrow\mathcal{C}_{\mathcal{U}}(\tau)$, the inclusion, for any oriented edge $[\sigma,\tau]$ of $bs(N\mathcal{U})$.
\end{itemize}
Note that vertex-ordering of simplices in $bs(N\mathcal{U})$ is given by reverse set-inclusion. It is easy to check that, when $X$ is a $G$-space and the cover $\mathcal{U}$ is $G$-invariant, $bs(N\mathcal{U})$ is a $G\text{-}\Delta$-complex, and $\mathcal{C}_{\mathcal{U}}$ becomes a $G$-complex of spaces over $bs(N\mathcal{U})$. In addition, the identification $\colim\mathcal{C}_{\mathcal{U}}=X$ preserves the $G$-actions.
}\end{example}

With the terminology reviewed above, the usual proof of the (topological non-equivariant) Nerve Lemma consists in showing that both maps in
\begin{equation}\label{estructurales}
X=\colim\mathcal{C}
_{\mathcal{U}}\stackrel{\;\;p_f}\longleftarrow\hocolim\mathcal{C}
_{\mathcal{U}}\stackrel{p_b}\longrightarrow\hocolim\widetilde{\mathcal{C}
_{\mathcal{U}}}=|N\mathcal{U}|
\end{equation}
are homotopy equivalences (see for instance~\cite[Corollary~4G.3]{MR1867354} and~\cite[Theorem~15.21]{MR2361455}). We apply the same strategy in the equivariant setting to show that, under the conditions of Theorem~\ref{conejura}, both $p_f$ (Theorem~\ref{projectionlemma} below) and $p_b$ (Proposition~\ref{homotopylemma} below) are $G$-homotopy equivalences, thus proving Theorem~\ref{conejura}. With a few adjustments (spelled out below), the explicit arguments parallel those in the non-equivariant case.

\begin{remark}\label{deeper}{\em
General properties in equivariant homotopy theory imply that the homotopy equivalence in~(\ref{hhdhd}) ---coming from an application of the (non-equivariant) Nerve Lemma--- is necessarily an $\mathcal{S}_n$-equivariant homotopy equivalence. Explicitely, Proposition~8.2.1 in~\cite{MR551743} asserts that a $G$-map that happens to be a (non-equivariant) homotopy equivalence (as is the case for the structural maps $p_f$ and $p_b$ in~(\ref{estructurales})), and that lands on a paracompact free $G$-space (as is the case for both $\F(|X|,n)$ and $|C(X,n)|$) would necessarily have to be a $G$-homotopy equivalence. Such an observation is consistent with the fact, noted in Remarks~\ref{semreg} and~\ref{prevented}, that the compatibility condition in Theorem~\ref{conejura} is vacuous in the configuration-space situation we just referenced. But the situation is certainly subtler in the presence of nontrivial isotropy subgroups.
}\end{remark}

\begin{proposition}[Equivariant Projection Lemma]\label{projectionlemma}
Let $G$ be a finite group acting on a paracompact space $X$. Let $\mathcal{U}=\{U_i\}_{i\in I}$ be a locally finite and $G$-invariant open cover of $X$. Then $p_f\colon\hocolim\mathcal{C}_{\mathcal{U}}\to\colim\mathcal{C}_{\mathcal{U}}=X$ is a $G$-homotopy equivalence.
\end{proposition}

\begin{proof}[Proof of Proposition~\ref{projectionlemma}] We have noted that $p_f$ is a $G$-equivariant map. It remains to construct a $G$-homotopy inverse. As shown in~\cite[Theorem~15.19]{MR2361455}, a (not necessarily equivariant) homotopy inverse $\ell\colon X\to\hocolim\mathcal{C}_{\mathcal{U}}$ for $p_f$ is encoded by the formula
\begin{equation}\label{encoded}
\ell_{\Psi}(x)=\left[ \left(\sum_i\psi_i(x)\cdot u_i,x\right) \right].
\end{equation}
Here square braces are used to indicate an equivalence class, while $\Psi=\{\psi_i\}_{i\in I}$ is a partition of unity subordinated to $\mathcal{U}$, and $u_i$ stands for the (geometric) vertex of $|N\mathcal{U}|$ realizing the (abstract) vertex $i$ of $N\mathcal{U}$. In order to get a \emph{$G$-equivariant} homotopy inverse, we use the standard trick of averaging $\Psi$, which yields a \emph{$G$-invariant} partition of unity. Explicitly, for each $i\in I$, consider the map $\phi_i\colon X\to[0,1]$ given by
\begin{equation}\label{promediado}
\phi_i(x)=\frac{1}{|G|}\sum_{g\in G}\psi_{gi}(gx).
\end{equation}
It is easy to check that the family $\Phi=\{\phi_i\}_{i\in I}$ is a partition of unity subordinated to $\mathcal{U}$ and that, in addition, it is $G$-invariant in the sense that
$$
\phi_i(x)=\phi_{gi}(gx), \text{ for any $(i,x,g)\in I\times X\times G$.}
$$
It follows that the map $\ell:=\ell_{\Phi}$ resulting in~(\ref{encoded}) from using $\Phi$ instead of $\Psi$ is $G$-equivariant, for
\begin{align*}
\ell(g\,{\cdot}\,x)&
=\left[ \left(\sum_i \phi_i(gx)\cdot u_i,gx\right) \right]
=\left[ \left(\sum_i \phi_{gi}(gx)\cdot u_{gi},gx\right) \right]\\
&=\left[ \left(\sum_i \phi_{gi}(gx)\cdot gu_{i},gx\right) \right]
=\left[ \left(\sum_i \phi_{i}(x)\cdot gu_{i},gx\right) \right]\\
&=g\cdot \left[ \left(\sum_i \phi_{i}(x)\cdot u_{i},x\right) \right]=g\,{\cdot}\,\ell(x).
\end{align*}
Since $\ell$ is clearly a right inverse for $p_f$, we only need to check that $\ell$ is a $G$-homotopy left inverse for $p_f$. But such a fact follows by noticing that a homotopy~$H$ from $\ell\circ p_f$ to the identity in $\hocolim\mathcal{C}_{\mathcal{U}}$ is given by ``linear deformation on fibers of $p_f$'' (see the proof of Theorem~15.19 in~\cite{MR2361455}). In more detail,
$$
H\left( \left[ \left(\sum t_i\cdot u_i,x\right)\right],t\right)=
\left[\left(\sum(t t_i+(1-t)\phi_i(x)) \cdot u_i,x\right)\right],
$$
so that
\begin{align*}
H&\left(g\cdot\left[\left(\sum t_i\cdot u_i,x\right)\right],t\right)=
H\left(\left[\left(\sum t_i\cdot u_{gi},gx\right)\right],t\right)
=H\left(\left[\left(\sum t_{\overline{g}i}\cdot u_{i},gx\right)\right],t\right)\\
&=\left[\left(\sum(t t_{\overline{g}i}+(1-t)\phi_i(g x)) \cdot u_i,g x\right)\right]
=\left[\left(\sum(t t_{\overline{g}i}+(1-t)\phi_{\overline{g}i}(x)) \cdot u_i,g x\right)\right]\\
&=\left[\left(\sum(t t_{i}+(1-t)\phi_{i}(x)) \cdot u_{gi},g x\right)\right]
=\left[\left(\sum(t t_{i}+(1-t)\phi_{i}(x)) \cdot gu_{i},g x\right)\right]\\
&=g\cdot \left[\left(\sum(t t_{i}+(1-t)\phi_{i}(x)) \cdot u_{i},x\right)\right]
=g\cdot H\left( \left[ \left(\sum t_i\cdot u_i,x\right)\right],t\right),
\end{align*}
where $\overline{g}$ stands for the inverse $g^{-1}$.
\end{proof}

\begin{remark}\label{softened}{\em
The finiteness assumption on $G$, used in~(\ref{promediado}), can be relaxed to assuming that $G$ is a compact topological group. In such a case~(\ref{promediado}) is replaced by the average
$$
\phi_i(x)=\frac{1}{\mu(G)}\int_{g\in G}\psi_{gi}(gx)d\mu,
$$
where $\mu$ is the Haar measure on $G$. 
}\end{remark}

In what follows $\mathcal{F}\colon\mathcal{C}_1\to\mathcal{C}_2$ is a $G$-map between $G$-complexes of spaces $\mathcal{C}_1$ and $\mathcal{C}_2$ over a $G\text{-}\Delta$-complex~$B$. Further, for each vertex $v$ of~$B$, we consider the isotropy subgroup $G_v=\{g\in G\colon gv=v\}$ of~$G$. By definition, $G_v$ acts on $\mathcal{C}_i(v)$, $i=1,2$, and $\mathcal{F}(v)\colon\mathcal{C}_1(v)\to\mathcal{C}_2(v)$ is a $G_v$-map.
\begin{proposition}[Equivariant Homotopy Lemma]\label{homotopylemma}
If $\mathcal{F}(v)$ is a $G_v$-homotopy equivalence for each vertex~$v$, then $\hocolim\mathcal{F}\colon\hocolim\mathcal{C}_1\to\hocolim\mathcal{C}_2$ is a $G$-homotopy equivalence.
\end{proposition}

The proof of the non-equivariant Homotopy Lemma (\cite[Theorem~15.12]{MR2361455}) is based on an inductive application of the Homotopy Extension Property for cofibrations. The same argument is adapted below to prove Proposition~\ref{homotopylemma} using the following standard facts in equivariant homotopy theory:

\begin{lemma}\label{equicof}
Consider a $G$-homotopy commutative diagram of $G$-spaces and $G$-maps
$$\xymatrix{Y\ar[r]^{p} & Z\\ A\ar[u]^{f_A}\ar@{^{(}->}[r] & X\ar[u]_h}$$
where $p$ is a $G$-homotopy equivalence and $A\hookrightarrow X$ is a $G$-cofibration. For every $G$-homotopy $H_A$ from $h_{|A}$ to $p\circ f_A$ there exists a $G$-map $f\colon X\to Y$ extending $f_A$, and a $G$-homotopy $H$ from $h$ to $p\circ f$ that extends $H_A$. In particular, if the $G$-cofibration $A\hookrightarrow X$ is a $G$-homotopy equivalence, then $A$ is in fact a strong $G$-deformation retract of $X$.
\end{lemma}
\begin{proof}
The first assertion is Proposition~8.2.2(a) in~\cite{MR551743}.
The second assertion, a standard consequence in the non-equivariant case, is obtained with $Y=A$, $Z=X$, $f_A$ and $h$ the identity maps, and $H_A$ the constant homotopy.
\end{proof}

\begin{corollary}\label{gretracto}
Let $M_f$ stand for the mapping cylinder of a $G$-map $f\colon X\to Y$. Then $f$ is a $G$-homotopy equivalence if and only if, under the standard inclusion $X\hookrightarrow M_f$, $X$ is a strong $G$-deformation retract of $M_f$.
\end{corollary}
\begin{proof}
As in the non-equivariant case, this is a direct consequence of Lemma~\ref{equicof}. Just note that the $G$-map $f\colon X \to Y$ can be thought of as a $G$-complex of spaces over the ($G$-trivial) $G\text{-}\Delta$-complex with two vertices joined by a single oriented edge. The homotopy colimit is then $M_f$ which, as noted in the preliminary considerations of this section, inherits a canonical $G$-action. Further, the canonical projection $M_f\to Y$ is clearly a $G$-homotopy equivalence, while the standard inclusion $X\hookrightarrow M_f$ is a $G$-cofibration.
\end{proof}

For convenience of exposition, we extract the following situation from the proof of Proposition~\ref{homotopylemma}:
\begin{lemma}\label{typical}
Let a group $G$ act transitively on a closed pair $(X,A)$. Assume there are subspaces $M_i$ ($i\in I$) of $X$ satisfying:
\begin{itemize}
\item The closure of $M_i$ is disjoint from $M_j$ if $i\neq j$.
\item $X-A\subseteq \bigsqcup_{i\in I} M_i$.
\item Each $g\in G$ sends each $M_i$ homeomorphically onto some $M_j$ (in which case we set $g\cdot i:=j$) so that the resulting action of $G$ on the index set $I$ is transitive.
\end{itemize}
Consider the isotropy subgroup $G_{i_0}=\{g\in G\colon g i_0=i_0\}$ for some arbitrarily chosen $i_0\in G$. If $M_{i_0}$ strong $G_{i_0}$-deformation retracts onto some $G_{i_0}$-invariant subspace $B_{i_0}$ containing $M_{i_0}\cap A$, then $X$ strong $G$-deformation retracts onto the $G$-invariant subspace $B:=\bigcup_{g\in G} g\cdot B_{i_0}$ (which contains $A$).
\end{lemma}
\begin{proof}
Fix a strong $G_{i_0}$-deformation $H_{i_0}\colon M_{i_0}\times[0,1]\to M_{i_0}$ of $M_{i_0}$ onto $B_{i_0}$. For $g\in G$ consider the homotopy $H_{i_0}^g$ defined through the commutative diagram
$$\xymatrix{M_{i_0}\times[0,1]\ar[d]_{g\times1}\ar[rr]^{\hspace{1cm}H_{i_0}} && M_{i_0}\ar[d]^g\\ M_{gi_0}\times[0,1]\ar[rr]_{\hspace{1cm}H_{i_0}^g} && M_{gi_0}.}$$
Note that $H_{i_0}^g$ is independent of $g$: If $g_1i_0=g_2i_0$ (so that $g_2^{-1}g_1\in G_{i_0}$) and $m\in M_{g_1i_0}=M_{g_2i_0}$, say with $g_1n_1=m=g_2n_2$ for $n_1,n_2\in M_{i_0}$ (so that $g_2^{-1}g_1n_1=n_2$), then $H_{i_0}(n_2,t)=H_{i_0}(g_2^{-1}g_1n_1,t)=g_2^{-1}g_1H_{i_0}(n_1,t)$, so that
$$H_{i_0}^{g_j}(m,t)=H_{i_0}^{g_j}(g_jn_j,t)=g_jH_{i_0}(n_j,t),$$
which is independent of $j\in\{1,2\}$.

We thus have homotopies $H_i\colon M_i\times [0,1]\to M_i$ fitting into commutative diagrams
$$\xymatrix{M_{i_0}\times[0,1]\ar[d]_{g\times1}\ar[rr]^{\hspace{1cm}H_{i_0}} && M_{i_0}\ar[d]^g\\ M_{i}\times[0,1]\ar[rr]_{\hspace{1cm}H_{i}} && M_{i}}$$
whenever $i=gi_0$ and, consequently, into commutative diagrams
$$\xymatrix{M_{i}\times[0,1]\ar[d]_{g\times1}\ar[rr]^{\hspace{1cm}H_{i}} && M_{i}\ar[d]^g\\ M_{gi}\times[0,1]\ar[rr]_{\hspace{1cm}H_{gi}} && M_{gi}.}$$
It is then straightforward to check that, together with the constant homotopy on $A$, the various $H_i$ assemble a homotopy $H\colon X\times [0,1]\to X$ giving the required strong $G$-deformation of $X$ onto $B$.
\end{proof}

\begin{proof}[Proof of Proposition~\ref{homotopylemma}]
We parallel the argument given in the proofs of~\cite[Theorem~15.12]{MR2361455} and~\cite[Proposition~4G.1]{MR1867354}. In view of Corollary~\ref{gretracto}, it suffices to show that the cylinder $\mathcal{M}:=M_{\hocolim\mathcal{F}}$ strong $G$-deformation retracts onto its ``non-glued'' end $\mathcal{H}_1:=\hocolim\mathcal{C}_1$. In doing so, the reader should keep in mind that $\mathcal{M}$ is canonically $G$-homeomorphic to the homotopy colimit of the $G$-complex of spaces obtained from the mapping cylinders of the several maps $\mathcal{F}(v)\colon\mathcal{C}_1(v)\to\mathcal{C}_2(v)$. Here and below, $v$ stands for a generic vertex of $B$.

The required (``global'') deformation is obtained as a concatenation of corresponding (``local'') deformations, each one associated to an inclusion $$\mathcal{H}_1\cup\mathcal{M}_{n-1}\hookrightarrow\mathcal{H}_1\cup\mathcal{M}_{n}$$ ($n\geq0$), where $\mathcal{M}_n$ is the portion of $\mathcal{M}$ lying over the $n$-th skeleton of $B$ (here we set $\mathcal{M}_{-1}:=\varnothing$).

We argue by induction on $n$. The case $n=0$ is straightforward, as $\mathcal{M}_0$ is just a disjoint union of mapping cylinders $M_{\mathcal{F}(v)}$. Since each $\mathcal{F}(v)$ is a $G_v$-homotopy equivalence, Corollary~\ref{gretracto} implies that the corresponding cylinder $M_{\mathcal{F}(v)}$ strong $G_v$-deformation retracts onto $M_{\mathcal{F}(v)}\cap \mathcal{H}_1$. Then, an (orbit-wise) application of Lemma~\ref{typical} yields the required strong $G$-deformation of $\mathcal{H}_1\cup\mathcal{M}_0$ onto $\mathcal{H}_1$.

The  case $n>0$ is similar although has a little simplifying twist at the end. Start by noticing that $(\mathcal{H}_1\cup\mathcal{M}_{n})-(\mathcal{H}_1\cup\mathcal{M}_{n-1})$ is contained in the topological disjoint union $\bigsqcup_\sigma \mathcal{M}_\sigma$, where $\sigma$ runs over the $n$-simplices of $B$, and $\mathcal{M}_\sigma$ stands for the part of $\mathcal{M}$ lying over the (combinatorial) interior of $\sigma$. Since $G$ acts on $B$ preserving the ordering of vertices on each simplex, there are only two posibilities for the action of a given $g\in G$ on a given $\sigma$: either g acts trivially on $\sigma$ or, else, $g$ takes $\sigma$ to a different simplex. Further, the former option holds precisely when $g$ lies in all the isotropy subgroups of the vertices of~$\sigma$. Thus, by Corollary~\ref{gretracto} and Lemma~\ref{typical}, it suffices to show that, for each $\sigma=[v_0,\ldots,v_n]$,
\begin{equation}\label{elpar}
\widetilde{\mathcal{M}}_\sigma \text{ strong $G_\sigma$-deformation retracts onto }\left( \widetilde{\mathcal{M}}_\sigma\cap\mathcal{H}_1 \right)\cup\bigcup_{i=0}^n\widetilde{\mathcal{M}}_{\sigma_i},
\end{equation}
where $G_\sigma:=\bigcap_{i=0}^n{G_{v_i}}$, $\sigma_i:=[v_0,\ldots,\widehat{\hspace{.4mm}v_i},\ldots,v_n]$ and, for a simplex $\tau$ of $B$, $\widetilde{\mathcal{M}}_\tau$ stands for the part of $\mathcal{M}$ lying over the (full) simplex~$\tau$. The standard presentation of the pair of spaces in~(\ref{elpar}) as an NDR-pair is $G_\sigma$-equivariant, so
\begin{equation}\label{inclcofi}
\left(\widetilde{\mathcal{M}}_\sigma\cap\mathcal{H}_1\right)\cup\bigcup_{i=0}^n\widetilde{\mathcal{M}}_{\sigma_i} \hookrightarrow \widetilde{\mathcal{M}}_\sigma
\end{equation}
is a $G_\sigma$-cofibration. Thus~(\ref{elpar}) follows from Corollary~\ref{gretracto} once we note that~(\ref{inclcofi}) is a $G_\sigma$-homotopy equivalence. The latter task is a direct consequence of the commutative diagram
$$\xymatrix{
\left(\widetilde{\mathcal{M}}_\sigma\cap\mathcal{H}_1\right)\cup\bigcup_{i=0}^n\widetilde{\mathcal{M}}_{\sigma_i} \ar@{^{(}->}[r]&
{\widetilde{\mathcal{M}}_{\sigma}} \ar[dd]_{\simeq} \\
{\widetilde{\mathcal{M}}_\sigma\cap\mathcal{H}_1 \ar@{^{(}->}[u]_{\simeq}\ar[d]^{\simeq}} & \\
{\mathcal{C}_1(v_n)} \ar@{^{(}->}[r]^{\simeq}& M_{\mathcal{F}(v_n)}
}$$
of $G_\sigma$-maps, where the vertical upward map is a $G_\sigma$-homotopy equivalence by induction, the two vertical downward maps are $G_{\sigma}$-homotopy equivalences by the construction of the homotopy colimit, and the bottom horizontal inclusion is a $G_{\sigma}$-homotopy equivalence by hypothesis and Corollary~\ref{gretracto}.
\end{proof}

%\bibliographystyle{plain}
%\bibliography{bib}

\medskip
{\small \sc \ 

Departamento de Matem\'aticas

Centro de Investigaci\'on y de Estudios Avanzados del I.P.N.

Av.~Instituto Polit\'ecnico Nacional n\'umero 2508

San Pedro Zacatenco, M\'exico City 07000, M\'exico

{\tt jesus@math.cinvestav.mx}

{\tt jgonzalez@math.cinvestav.mx}}
\end{document}